\documentclass[11pt]{article}
\usepackage{amsfonts}
\usepackage{amssymb,amsmath}
\usepackage{amsthm}
\usepackage{latexsym}
\usepackage{graphics}
\usepackage{epic}
\usepackage{epsfig}
\usepackage{psfrag}
\usepackage{bbm}
\usepackage{dsfont}
\usepackage{enumitem}
\usepackage{stmaryrd}
\usepackage{xcolor}
\usepackage{hyperref} 
\usepackage{soul} 

\numberwithin{equation}{section}
\def\PP{\mathbb{P}}

\def\RR{\mathbb{R}}
\def\NN{\mathbb{N}}

\def\EE{\mathbb{E}}
\def\11{\mathbbm{1}}

\def\E{\mathbb{E}}
\def\P{\mathbb{P}}
\def\R{\mathbb{R}}

\def\N{\mathbb{N}}

\def\d{\partial}
\def\Z{\mathbb{Z}}
\def\ZZ{\mathbb{Z}}

\newtheorem{thm}{Theorem}[section]

\newtheorem{cor}[thm]{Corollary}

\newtheorem{prop}[thm]{Proposition}

\theoremstyle{remark}
\newtheorem{rem}{Remark}
\newtheorem{exa}{Example}

\begin{document}

\title{Population processes with unbounded extinction rate conditioned to non-extinction}

\author{Nicolas Champagnat$^{1,2,3}$, Denis Villemonais$^{1,2,3}$}

\footnotetext[1]{Universit\'e de Lorraine, IECL, UMR 7502, Campus Scientifique, B.P. 70239,
  Vand{\oe}uvre-l\`es-Nancy Cedex, F-54506, France}
\footnotetext[2]{CNRS, IECL, UMR 7502,
  Vand{\oe}uvre-l\`es-Nancy, F-54506, France} \footnotetext[3]{Inria, TOSCA team,
  Villers-l\`es-Nancy, F-54600, France.\\
  E-mail: Nicolas.Champagnat@inria.fr, Denis.Villemonais@univ-lorraine.fr}

\maketitle

\begin{abstract}
  This article studies the quasi-stationary behaviour of population processes with unbounded absorption rate, including
  one-dimensional birth and death processes with catastrophes and multi-dimensional birth and death processes, modeling biological populations in interaction. 
  To handle this situation, we develop original non-linear Lyapunov criteria.
  We obtain the exponential convergence in total variation of the conditional distributions to a unique quasi-stationary
  distribution, uniformly with respect to the initial distribution. Our results cover all one-dimensional birth and death processes
  which come down from infinity with catastrophe rate satisfying appropriate bounds, and multi-dimensional birth and death models with
  stronger intra-specific than inter-specific competition.
\end{abstract}



\section{Introduction}
\label{sec:intro}

This article is devoted to the study of the quasi-stationary behavior of continuous time Markov processes $(X_t)_{t\geq 0}$ in a discrete state space $E$ almost surely absorbed in finite time at some cemetery point $\d$. This means that  $X_t=\d$ for all $t\geq\tau_\d$ and $\P_x(\tau_\d<\infty)=1$ for all $x\in E$, where $\tau_\d=\inf\{t\geq 0,\ X_t=\d\}$ is the absorption time of the process. We recall that $\alpha$ is a quasi-stationary distribution if it is a probability measure on $E$ such that
$$
\PP_\alpha(X_t\in\cdot\mid t<\tau_\d)=\alpha,\quad\forall t\geq 0.
$$
In Section~\ref{sec:lyap}, our goal is to provide a tractable criterion ensuring the existence of a unique quasi-stationary
distribution $\alpha$ on $E$ such that, for all probability measure $\mu$ on $E$,
\begin{equation}
  \label{eq:conv}
  \left\|\PP_\mu(X_t\in\cdot\mid t<\tau_\d)-\alpha\right\|_{TV}\leq C e^{-\lambda t},\quad t\geq 0,
\end{equation}
where $C,\gamma$ are positive constants and $\|\cdot\|_{TV}$ is the usual total variation distance. Our criterion is based on the
computation of the infinitesimal generator of the process and Lyapunov-type functions. It takes the form of the following non-linear
Foster-Lyapunov criterion : there exists a bounded function $V:E\rightarrow\R_+$ and a norm-like function $W:E\rightarrow\R_+$ such
that, for any probability measure $\mu$ on $E$,
\begin{equation}
  \label{eq:lyap-intro}
  \mu(LV)-\mu(V)\mu(L\11_E)\leq A-B\mu(W),
\end{equation}
for some constants $A,B>0$ and where $L$ is the generator of the process $X$ (see Section~\ref{sec:lyap} for details).

Many properties can be deduced from~\eqref{eq:conv}. First, regardless of the initial condition, the quasi-stationary distribution describes
the state of the population when it survives for a long time. One of the most notable features of quasi-stationary populations is the
existence of a so-called \emph{mortality/extinction plateau}: there exists $\lambda_0>0$ limit of the extinction rate of the
population (see~\cite{meleard-villemonais-12}). The constant $-\lambda_0$ is actually the largest non-trivial eigenvalue of the
generator $L$ and satisfies
$$
\PP_\alpha(t<\tau_\d)=e^{-\lambda_0 t},\quad\forall t\geq 0.
$$
In addition,~\eqref{eq:conv} implies that $x\mapsto e^{\lambda_0 t}\PP_x(t<\tau_\d)$ converges uniformly to a function $\eta:E\rightarrow(0,+\infty)$ when $t\rightarrow+\infty$~\cite[Theorem~2.1]{ChampagnatVillemonais2016U}. Moreover, $\eta$ is the eigenfunction of $L$ corresponding to the eigenvalue $\lambda_0$~\cite[Prop.\,2.3]{ChampagnatVillemonais2014}. It also implies  the existence and the exponential ergodicity of the associated $Q$-process, defined as the process $X$ conditionned to never be
extinct (see~\cite[Thm.\,3.1]{ChampagnatVillemonais2014} for a precise definition). The convergence of the conditional laws of $X$ to the law of the $Q$-process holds also uniformly in total variation norm~\cite[Theorem~2.1]{ChampagnatVillemonais2016U}, which entails conditional ergodic properties~\cite[Corollary~2.2]{ChampagnatVillemonais2016U}.

\bigskip The second part of the paper is devoted to the application of the criterion~\eqref{eq:lyap-intro} to birth and death
processes with catastrophes in dimension~1 or more.

In Section~\ref{sec:1D}, we prove that uniform exponential convergence to a unique quasi-stationary distribution holds for all one-dimensional birth and death processes coming down from infinity with catastrophe rate satisfying appropriate bounds. It was already known that~\eqref{eq:conv} holds if and only if the process comes down from infinity for birth and death processes without catastrophes~\cite{vanDoorn1991,Martinez-Martin-Villemonais2012} and birth and death processes with bounded catastrophe rates~\cite{vanDoorn2012,ChampagnatVillemonais2016}. Our result only assumes suitable growth conditions of the oscillations of the catastrophe rate. This result can be easily transfered to multi-type birth and death processes dominated by one-dimensional birth and death processes coming down from infinity (see Section~\ref{sec:multi-d-dom}).

In Section~\ref{sec:multi-d-strong-intra}, we focus on multi-type birth and death processes with logistic competition (also known as
competitive Lotka-Volterra interaction), absorbed when one of the coordinates hits $0$ or when a catastrophe occurs. This case is
critical with respect to the previous method in the sense that the absorption rate when the process is close to the boundary (i.e.\
when one of the coordinates of the process is $1$) does not satisfy, but nearly satisfies, the bounds obtained in
Section~\ref{sec:multi-d-dom}. Using a different Lyapunov function, we are able to prove that~\eqref{eq:conv} holds provided that the
intra-specific competition is stronger than the inter-specific competition, in the sense that the competition coefficient between
individuals of different species is much smaller than the competition coefficient between individuals of the same species. In Biology, this
assumption is particularly relevant to model a multi-type population where individuals survive by consuming type-specific resources.
This is for example the case for populations of bacteria in a chemostat which are specialized for consuming different
resources~\cite{champagnat-jabin-meleard-2014}.

Quasi-stationary distributions for population processes have received much interest in the recent years (see the
surveys~\cite{meleard-villemonais-12,vanDoorn2013} and the book~\cite{collet-martinez-al-13b}). The specific question of estimates on
the speed of convergence to quasi-stationary distributions for one dimensional birth and death processes has been studied
in~\cite{diaconis-miclo-09,chazottes-al-15}. The multi-dimensional situation, which takes into account the existence of several types
of individuals in a population, is much less understood, except in the branching case of multi-type Galton-Watson processes
(see~\cite{Athreya1972,penisson-11}) and for specific cooperative models with bounded absorption rate
(see~\cite{ChampagnatVillemonais2014}). For results on the quasi-stationary behaviour of continuous-time and continuous state space
models, we refer to~\cite{CCLMMS09,Littin2012,ChampagnatVillemonais2015} for the one dimensional case and
to~\cite{Pinsky1985,Gong1988,Cattiaux2008,KnoblochPartzsch2010, DelMoralVillemonais2015, ChampagnatCoulibalyVillemonais2015} for the
multi-dimensional situation. Infinite dimensional models have been studied in~\cite{Collet2011,ChampagnatVillemonais2014}. Several
papers studying the quasi-stationary behaviour of models applied to biology, chemistry, demography and finance are listed
in~\cite{Pollett}.

\section{General Lyapunov citerion for exponential convergence of conditional distributions}
\label{sec:lyap}

We consider a continuous-time Markov process $(X_t,t\geq 0)$ on a state space $E\cup\{\d\}$ with $E$ denumerable, almost surely
absorbed in finite time in $\d$. We denote by $q_{x,y}\geq 0$ the jump rate of $X$ from $x\in E\cup\{\d\}$ to $y\neq x$ and recall
that $\sum_{y\not=x}q_{x,y}<\infty$ for all $x\in E\cup\{\d\}$ is needed for the Markov process $X$ to be well-defined. The fact that
$\d$ is absorbing means that $q_{\d,x}=0$ for all $x\in E$. We set $\tau_\d=\inf\{t\geq 0:X_t=\d\}$. We use the usual notation
$\P_x=\P(\cdot\mid X_0=x)$ and for all probability measure $\mu$ on $E\cup\{\d\}$, $\P_\mu=\int_{E\cup\{\d\}}\P_x\mu(dx)$.

We assume that it is irreductible away from $\d$ in the sense that:
\begin{equation}
  \label{hyp:Lyap}
  \mbox{For all }x,y\in E,\ \P_x(X_1=y)>0.
\end{equation}
We recall that, by classical properties of discrete Markov processes in continuous time, the time $t=1$ in this assumption is
arbitrary and could be replaced by any other positive value.

We denote by $L$ its generator, defined for all bounded
$\phi:E\cup\{\d\}\rightarrow \R$ as
$$
L\phi(x)=\sum_{y\in E\cup\{\d\},\ y\neq x}q_{x,y}[\phi(y)-\phi(x)],\quad\forall x\in E\cup\{\d\}.
$$
Note that we use here a definition a bit more general than the usual (standard or weak) infinitesimal generator(s), since we do not
require $L\phi$ to be bounded. We extend the last definition to bounded functions $\phi:E\rightarrow \R$ defined only on $E$ as
$$
L\phi(x)=\sum_{y\in E,\ y\neq x}q_{x,y}[\phi(y)-\phi(x)],\quad\forall x\in E.
$$
Note that $L\11_E(x)\leq 0$ for all $x\in E$ and $L\11_E\not\equiv 0$ since $\tau_\d<\infty$ a.s.

The next result gives our non-linear Lyapunov criterion. We will say that a function $W:E\rightarrow\mathbb{R}$ is \emph{norm-like}
if $\{x\in E:W(x)\leq a\}$ is finite for all $a\in\mathbb{R}$. Equivalently, this means that $W$ converges to $+\infty$ out of finite
subsets of $E$, in the sense that
$$
\lim_{n\rightarrow+\infty} \inf_{x\not\in K_n} W(x)=+\infty
$$
for all increasing sequence $(K_n)_{n\geq 1}$ of subsets of $E$ such that $\cup_n K_n=E$.

\begin{thm}
  \label{thm:Lyap}
  Assume that there exist a bounded function $V:E\rightarrow\mathbb{R}_+$ such that $LV$ is bounded from above and a norm-like
  function $W:E\rightarrow\mathbb{R}$ such that, for any probability measure $\mu$ on $E$,
  \begin{equation}
    \label{eq:Lyap}
    \mu(LV)-\mu(V)\mu(L\11_E)\leq A-B\mu(W),
  \end{equation}
  for some constants $A,B>0$. Then, assuming~\eqref{hyp:Lyap}, the process $X$ admits a unique quasi-stationary distribution
  $\nu_{QSD}$ and there exist constants $C,\gamma>0$ such that for all probability measure $\mu$ on $E$,
  \begin{equation}
    \label{eq:cv-expo}
    \|\P_\mu(X_t\in\cdot\mid <\tau_\d)-\nu_{QSD}\|_{TV}\leq C e^{-\gamma t},\quad\forall t\geq 0,
  \end{equation}
  where $\|\cdot\|_{TV}$ is the total variation norm on finite signed measures, defined by $\|\mu\|_{TV}=\sup_{f\in L^\infty(E),\
    \|f\|_\infty\leq 1}|\mu(f)|$.
\end{thm}

The question of existence of a quasi-stationary distribution for similar models can be tackled with the theory of $R$-positive
matrices~\cite{FerrariKestenMartinez1996,gosselin-01}. However, these results do not give uniqueness of the quasi-stationary
distribution, nor the uniform convergence of conditional distributions to the quasi-stationary distribution. In particular, they do
not provide all the properties deduced from~\eqref{eq:conv} in~\cite{ChampagnatVillemonais2016U,ChampagnatVillemonais2014} and
mentioned in the introduction. This is also an important issue for applications, since the initial distribution of the population is
usually unknown and, in cases where there is no uniform convergence, the convergence rate of conditional distributions usually highly
depends on the tail of the initial distribution (see the discussion in~\cite{MV12}).

The proof of Theorem~\ref{thm:Lyap} is based on the next Proposition.

\begin{prop}
  \label{prop:mu_t}
  Fix $x\in E$ and let $\mu_t(\cdot)=\PP_x(X_t\in\cdot\mid t<\tau_\d)$. Let $V:E\rightarrow\RR_+$ be a bounded function such that
  $LV$ is bounded from above. Then, for all $t\geq 0$,
  \begin{equation}
    \label{eq:prop-mu_t}
    \mu_t(V)=V(x)+\int_0^t\Big[\mu_s(LV)-\mu_s(V)\mu_s(L\mathbbm{1}_E)\Big]\,ds,    
  \end{equation}
  where the value of the integral in the r.h.s.\ is well-defined since, for all $t\geq 0$,
  $$
  \mu_s(LV)-\mu_s(V)\mu_s(L\mathbbm{1}_E)\in L^1([0,t]).
  $$
\end{prop}

\begin{proof}
  Fix $(K_n)_{n\geq 1}$ an increasing sequence of finite subsets of $E$ such that $\cup_n K_n=E$, and for all $n\geq 1$, let
  $\tau_n:=\inf\{t\geq 0:X_t\not\in K_n\}$. We define $X^n$ as the process $X$ stopped at time $\tau_n$, and denote by $L^n$ its
  infinitesimal generator, given by $L^n f(x)=Lf(x)\mathbbm{1}_{x\in K_n}$. In particular, $L^n V$ is bounded and $V$ belongs to the
  domain of the (standard) infinitesimal generator of $X^n$, hence and Dynkin's formula applies and entails
  \begin{equation}
    \label{eq:prop1-0}
    \EE_x V(X^n_t)=V(x)+\int_0^t\EE_x\left[L^nV(X^n_s)\right]\,ds,
  \end{equation}
  with the convention that $V(\d)=0$. Letting $n\rightarrow+\infty$, Lebesgue's theorem applied to the left-hand side and Fatou's
  lemma applied to the right-hand side imply that
  \begin{equation*}
    \EE_x V(X_t)\leq V(x)+\int_0^t\EE_x\left[LV(X_s)\right]\,ds.
  \end{equation*}
  Since $V\geq 0$ and $LV$ is bounded from above, we deduce that $\EE_x LV(X_s)\in L^1([0,t])$. Therefore, using the equality
  $L^nV(X^n_s)=\mathbbm{1}_{s<\tau_n}LV(X_s)$, we can actually apply Lebesgue's Theorem to the right-hand side of~\eqref{eq:prop1-0}
  and hence
  \begin{equation}
    \label{eq:prop1}
    \EE_x V(X_t)=V(x)+\int_0^t\EE_x\left[LV(X_s)\right]\,ds.
  \end{equation}
  The same argument applies to $\11_E$, hence
  $$
  \PP_x(t<\tau_\d)=1+\int_0^t\EE_x\left[L\mathbbm{1}_E(X_s)\right]\,ds.
  $$
  Therefore (cf.\ e.g.~\cite[Thm.~VIII.2 and Lem.~VIII.2]{brezis}), for all $T>0$, $t\mapsto \EE_x V(X_t)$ and $t\mapsto
  \PP_x(t<\tau_\d)$ belong to the Sobolev space $W^{1,1}([0,T])$ (the set of functions from $[0,T]$ to $\mathbb{R}$ in $L^1$
  admitting a derivative in the sense of distributions in $L^1$). Since $\P_x(T<\tau_\d)>0$, we deduce from standard properties of
  $W^{1,1}$ functions~\cite[Cor.~VIII.9 and Cor.~VIII.10]{brezis} that $t\mapsto \EE_x V(X_t)/\PP_x(t<\tau_\d)=\mu_t(V)$ belongs to
  $W^{1,1}$ and admits as derivative
  $$
  t\mapsto\frac{\EE_x LV(X_t)}{\PP_x(t<\tau_\d)}-\EE_x V(X_t)\frac{\EE_x L\11_E(X_t)}{\PP_x(t<\tau_\d)^2}=
  \mu_t(LV)-\mu_t(V)\mu_t(L\mathbbm{1}_E)\in L^1([0,T]).
  $$
  Hence we have proved~\eqref{eq:prop-mu_t}.
\end{proof}

\begin{proof}[Proof of Theorem~\ref{thm:Lyap}]
The proof is divided into two steps, consisting in proving that the two conditions (A1) and (A2) below hold true.
It has been proved in~\cite{ChampagnatVillemonais2014} that (A1) and (A2) imply~\eqref{eq:cv-expo}.

There exists a probability measure $\nu$ on $E$ such that
\begin{itemize}
\item[(A1)] there exist $t_0,c_1>0$ such that for all $x\in E$,
  $$
  \PP_x(X_{t_0}\in\cdot\mid t_0<\tau_\d)\geq c_1\nu(\cdot);
  $$
\item[(A2)] there exists $c_2>0$ such that for all $x\in E$ and $t\geq 0$,
  $$
  \PP_\nu(t<\tau_\d)\geq c_2\PP_x(t<\tau_\d).
  $$
\end{itemize}

\noindent\textit{Step 1: Proof of~(A1).}\\
Assume that $X_0=x$ and defne $\mu_t$ as in Prop~\ref{prop:mu_t}. Assumption~\eqref{eq:Lyap} and Prop.~\ref{prop:mu_t} imply that
$$
\mu_t(V)\leq \|V\|_\infty+At-B\int_0^t \mu_s(W)\,ds.
$$
Since $V\geq 0$, there exist $s_0\leq \|V\|_\infty/A$ such that
$$
\mu_{s_0}(W)\leq 2A/B.
$$
Let us define the set $K_0= \{x\in E:W(x)< 4A/B\}$ which is finite since $W$ is norm-like. Using the previous inequality and Markov's
inequality, we obtain that $\mu_{s_0}(K_0)\geq 1/2$.

Fix $x_0\in K_0$. The minimum 
$$
p=\min_{x\in K_0} \P_x\{X_u=x_0,\ \forall u\in[\|V\|_\infty/A, 2 \|V\|_\infty/A)\}
$$
is positive because $K_0$ is finite, $X$ is irreducible away from $\d$ and the jumping rate from $x_0$ is finite. Hence
\begin{align*}
\mu_{s_0}\Big(\P_\cdot\{X_u=x_0,\ \forall u\in[\|V\|_\infty/A, 2 \|V\|_\infty/A)\}\Big)\geq \frac{p}{2} >0,
\end{align*}
  Using the Markov property, we deduce that
\begin{align*}
\mu_{2 \|V\|_\infty/A}\{x_0\}&=\frac{\E_x[\11_{s_0<\tau_\d}\P_{X_{s_0}}(X_{2 \|V\|_\infty/A-s_0}=x_0)]}{\P_x(2 \|V\|_\infty/A<\tau_\d)} \\ &
=\mu_{s_0}(\P_\cdot(X_{2 \|V\|_\infty/A-s_0}=x_0))\frac{\P_x(s_0<\tau_\d)}{\P_x(2 \|V\|_\infty/A<\tau_\d)}\\
&\geq \mu_s\Big(\P_\cdot\{X_u=x_0,\ \forall u\in[\|V\|_\infty/A, 2 \|V\|_\infty/A)\}\Big)\\
&\geq \frac{p}{2}>0.
\end{align*}
Note that $s_0$ may depend on the initial value $x$ of the process, but since $p$ does not depend on $x$, we have indeed proved that
(A1) is satisfied with $\nu=\delta_{x_0}$, $t_0=2 \|V\|_\infty/A$ and $c_1=p/2$.

\medskip
\noindent\textit{Step 2: Proof of~(A2).}\\
Since $L\11_E\leq 0$,~\eqref{eq:Lyap} applied to $\mu=\delta_x$ implies that
\begin{align*}
  LV(x)\leq A-BW(x),\quad\forall x\in E.
\end{align*}
Hence, using the same argument as for the proof of~\eqref{eq:prop1}, for all finite subset $K$ of $E$, denoting by $\tau_{K}$ the
first hitting time of $K$ by the process $X$, we have for all $x\in E\setminus K$
\begin{equation*}
  \EE_x V(X_{t\wedge\tau_K})=V(x)+\EE_x\left[\int_0^{t\wedge\tau_K}LV(X_s)\,ds\right],
\end{equation*}
with the convention that $V(\d)=0$. We deduce that
$$
0\leq\E_x[V(X_{\tau_{K}\wedge\tau_\d})]\leq \|V\|_\infty+\left(A-B\inf_{y\not\in K}W(y)\right)\E_x(\tau_K\wedge\tau_\d).
$$
Since $K$ is arbitrary and $W$ is norm-like, for all $\varepsilon>0$, there exists $K\subset E$ finite such that
$\E_x(\tau_{K}\wedge\tau_\d)\leq\varepsilon$ for all $x\in E\setminus K$. Hence Markov's inequality implies that
$$
\sup_{x\in E\setminus K}\P_x(\tau_{K}\wedge\tau_\d\geq 1)\leq\varepsilon,
$$
and Markov's property then entails
$$
\sup_{x\in E\setminus K}\P_x(\tau_{K}\wedge\tau_\d\geq n)\leq\varepsilon^n,\quad\forall n\in\mathbb{N}.
$$
Since $\varepsilon$ was arbitrary, we have proved that there exists a finite $K\subset E$ such that
\begin{equation*}
  M:=\sup_{x\in E\setminus K}\E_x[\exp(q_{x_0}(\tau_{K}\wedge\tau_\d))]<\infty, 
\end{equation*}
where $x-0$ was fixed in Step 1 and $q_{x_0}=\sum_{y\neq x_0}q_{x_0,y}$ is the total jump rate of the process $X$ from $x_0$. We may
(and will) assume without loss of generality that $x_0\in K$.
The irreducibility of $X$ and the finiteness of $K$ entail the existence of a constant $C>0$ independent of $t$ such that
\begin{equation*}
  \sup_{x\in K}\P_x(t<\tau_\d)\leq C\inf_{x\in K}\P_x(t<\tau_\d)\leq C\P_{x_0}(t<\tau_\d),\quad\forall t\geq 0.
\end{equation*}
By definition of $\lambda$ and by the Markov property, for all $0\leq s\leq t$,
\begin{align*}
  e^{-q_{x_0} s} \P_{x_0}(t-s<\tau_\d)& =\P_{x_0}(X_u=x_0,\forall u\in [0,s])\P_{x_0}(t-s<\tau_\d)\\
  &\leq \P_{x_0}(t<\tau_\d).
\end{align*}

Using the last three inequalities and the strong Markov property, we have for all $x\in E$,
\begin{align*}
  \P_x(t<\tau_\d) & =\P_x(t<\tau_K\wedge\tau_\d)+\P_x(\tau_K\wedge\tau_\d\leq t<\tau_\d)\\
  &\leq Me^{-q_{x_0} t}+\int_0^t \sup_{y\in K\cup\{\d\}}\P_y(t-s<\tau_\d)\P_y(\tau_K\wedge\tau_\d\in ds)\\
  &\leq Me^{-q_{x_0} t}+ C\int_0^t \P_{x_0}(t-s<\tau_\d)\P_x(\tau_K\wedge\tau_\d\in ds) \\
  &\leq M\PP_{x_0}(t<\tau_\d)+ C\P_{x_0}(t<\tau_\d)\int_0^t e^{q_{x_0} s}\P_x(\tau_K\wedge\tau_\d\in ds)\\
  &\leq M(1+C)\P_{x_0}(t<\tau_\d).
\end{align*}
This entails (A2) for $\nu=\delta_{x_0}$. 
\end{proof}

\section{One-dimensional birth and death processes with unbounded catastrophe rates}
\label{sec:1D}

The quasi-stationary behaviour of one-dimensional birth and
death processes with bounded catastrophe rates has been  studied
in~\cite{vanDoorn2012,ChampagnatVillemonais2014}. We study in this section the implication of the non-linear Lyapunov criterion of Theorem~\ref{thm:Lyap} on  birth and death processes with possibly unbounded catastrophe rates.

The Markov process $(X_n,n\geq 0)$ in $E\cup\{\d\}$ with $E=\mathbb{N}:=\{1,2,\ldots\}$ and $\d=0$ is a birth and death process with
catastrophe if $0$ is absorbing and if it jumps from state $k\in E$ to $k+1$ at rate $b_k$; to $k-1$ at rate $b_k$ and to $0$
at rate $a_k$ for some positive sequences $(b_k)_{k\geq 1}$ and $(d_k)_{k\geq 1}$ and a nonnegative sequence $(a_k)_{k\geq
  1}$. In other words, for all $i\neq j$ in $\mathbb{Z}_+$,
$$
q_{i,j}=
\begin{cases}
  b_i & \mbox{if }j=i+1\mbox{ and }i\geq 1, \\
   d_i & \mbox{if }j=i-1\mbox{ and }i\geq 2, \\ 
   a_i & \mbox{if }j=0\mbox{ and }i\geq 2, \\
   a_1+ d_1 & \mbox{if }j=0\mbox{ and }i=1, \\ 
  0 & \mbox{otherwise}.
\end{cases}
$$

We assume that
\begin{equation}
  \label{eq:S}
  S:=\sum_{n\geq
    1}\left(\frac{1}{ d_n}+\frac{b_n}{ d_n d_{n+1}}+\ldots+\frac{b_n\ldots b_{n+k}}{ d_n\ldots d_{n+k+1}}
    +\ldots\right)<+\infty.
\end{equation}
We may write $S$ as
$$
S=\sum_{n\geq 1}\frac{1}{ d_n\pi_n}\sum_{k\geq n}\pi_k=\sum_{k\geq 1}\pi_k\sum_{n=1}^k\frac{1}{ d_n\pi_n},
$$
where $\pi_1=1$ and, for all $k\geq 2$,
$$
\pi_k=\frac{b_1\ldots b_{k-1}}{ d_1\ldots d_k}.
$$
This condition implies that the birth and death process without catastrophe ($a_k=0$ for all $k$) is well-defined, and is equivalent
to the fact that the process comes down from infinity (i.e.\ $+\infty$ is an entrance
boundary)~\cite{bansaye-meleard-richard-15,anderson-91}.

\begin{thm}
  \label{thm:PNM-1D}
  Under the assumption~\eqref{eq:S}, consider a nondecreasing and unbounded function $W:\mathbb{N}\rightarrow\mathbb{R}_+$ such that
  \begin{equation}
    \label{eq:condition-W}
    \sum_{n\geq 1}\frac{1}{ d_n\pi_n}\sum_{k\geq n}W(k)\pi_k=\sum_{k\geq 1}W(k)\pi_k\sum_{n=1}^k\frac{1}{ d_n\pi_n}<+\infty.
  \end{equation}
  If $ a_k=\kappa_k+o(W(k))$ when $k\rightarrow+\infty$ where $(\kappa_k)_{k\geq 1}$ is any non-decreasing sequence,
  then~\eqref{eq:Lyap} is satisfied for the function $W$ and
  $$
  V(x)=\sum_{n=1}^{x+1}\frac{1}{ d_n\pi_n}\sum_{k\geq n}W(k)\pi_k.
  $$
  In particular, the conclusion of Theorem~\ref{thm:Lyap} holds true.
\end{thm}

\begin{rem}
  \begin{enumerate}
  \item Since $S<\infty$, De la Vall\'ee-Poussin's classical result entails that we can always find an unbounded non-decreasing
    function $W$ satisfying~\eqref{eq:condition-W}.
  \item In particular, we can always find $\widetilde{W}$ satisfying the same condition and such that $W(k)=o(\widetilde{W}(k))$.
    Therefore, if we relax our assumption on $ a$ as $ a_k=\kappa_k+O(W(k))$, the conclusion of Theorem~\ref{thm:Lyap} will still
    hold true.
  \item Note that no assumption of boundedness or limited growth of $\kappa_k$ as $k\rightarrow+\infty$ is required.
  \item In fact, the assumption $ a_k=\kappa_k+O(W(k))$ concerns the fluctuations of the sequence $(a_k)_{k\in\N}$. More precisely, setting
    $$
    \kappa^+_k=\sup_{\ell\leq k} a_\ell.
    $$
    and
    $$
    \kappa^-_k=\inf_{\ell\geq k} a_\ell,
    $$
    one easily obtains the following corollary.
  \end{enumerate}
\end{rem}

\begin{cor}
  \label{cor:PNM-1D}
  Under the assumptions~\eqref{eq:S} and
  \begin{equation}
    \sum_{k\geq 1}(\kappa^+_k-\kappa^-_k)\pi_k\sum_{n=1}^k\frac{1}{ d_n\pi_n}<+\infty,
  \end{equation}
  the conclusion of Theorem~\ref{thm:Lyap} holds true.
\end{cor}

\begin{exa}
  \label{ex:exa-log-1D}
  The logistic birth and death process is commonly used in biological applications as a population model with competition or limited
  carrying capacity. A jump from a state $k\geq 2$ to $0$ corresponds to a catastrophic event where the whole population is
  annihilated. The birth and death rates are given here by $b_k=bk$ and $ d_k=dk+ck(k-1)$ for some $b,c,d>0$. It is well-known that
  $S<\infty$ in this case. More precisely, we have
  $$
  u_n:=\frac{1}{ d_n\pi_n}=\frac{d(d+c)\ldots(d+c(n-2))}{b^{n-1}}.
  $$
  Since the sequence $(u_n)_{n\geq 1}$ is strongly diverging to $+\infty$, $\sum_{n=1}^k u_n$ is equivalent to its last term $u_k$
  when $k\rightarrow+\infty$. This can be proved as follows: for all $\varepsilon>0$, there exists $n_0$ such that, for all $n\geq
  n_0$, $u_n\leq \varepsilon u_{n+1}$. Hence $u_n\leq u_k\varepsilon^{k-n}$ for all $n_0\leq n\leq k$ and
  $$
  u_k\leq \sum_{n=1}^ku_n\leq \sum_{n=1}^{n_0}u_n+u_k\sum_{n=n_0+1}^k \varepsilon^{k-n}\leq
  \sum_{n=1}^{n_0}u_n+\frac{u_k}{1-\varepsilon},
  $$
  which yields $\sum_{n=1}^k u_n\sim u_k$ when $k\rightarrow=\infty$. Therefore,
  $$
  \pi_k \sum_{n=1}^k \frac{1}{ d_n\pi_n}\sim \frac{\pi_k}{ d_k\pi_k}= d_k\sim\frac{1}{ck^2}
  $$
  when $k\rightarrow+\infty$. Hence we can choose in Theorem~\ref{thm:PNM-1D} any norm-like function $W:E\rightarrow\RR_+$ such that
  $\sum_{k\geq 1}\frac{W(k)}{k^2}<\infty$. For example, the conclusions of Theorem~\ref{thm:Lyap} are true if $| a_k-\kappa_k|=
  O(k^{1-\varepsilon})$ for some $\varepsilon>0$ and some non-decreasing non-negative sequence $(\kappa_k)_{k\in\N}$.
\end{exa}

\begin{rem}
\label{rem:arg-log-gener}
  Note that the argument in the last example would work for any birth and death process which comes down from infinity and such that
  $b_k=o( d_k)$, so that any function $W$ such that $\sum_{k\geq 1}W(k)/ d_k<\infty$ would satisfy the conditions of
  Theorem~\ref{thm:Lyap}.
\end{rem}

\begin{exa}
  We consider now an example where $b_k\neq o( d_k)$ and, more precisely, the case where the birth and death process $X$ without
  catastrophe is a local martingale, i.e.\ $b_k=d_k$ for all $k\geq 1$. In this case,
  $$
  \pi_k \sum_{n=1}^k \frac{1}{ d_n\pi_n}=k\pi_k=\frac{k}{ d_k}.
  $$
  Hence the birth and death process comes down from infinity if and only if $\sum_{k\geq 1}k/ d_k<\infty$ and any norm-like function
  $W$ such that $\sum_{k\geq 1}kW(k)/ d_k$ would satisfy the conditions of Theorem~\ref{thm:Lyap}. For example, if $b_k= d_k=k^3$,
  the conclusions of Theorem~\ref{thm:Lyap} are true if $| a_k-\kappa_k|=O(k^{1-\varepsilon})$ for some $\varepsilon>0$ and some
  non-decreasing non-negative sequence $(\kappa_k)_{k\in\N}$.
\end{exa}

\begin{proof}[Proof of Theorem~\ref{thm:PNM-1D}]
  We observe that the generator of the birth and death process with catastrophes can be written, for all bounded
  $\phi:\mathbb{Z}_+\rightarrow\RR$ such that $\phi(0)=0$, as
  $$
  L\phi(x)=L_0\phi(x)- a_x\phi(x),
  $$
  where
  $$
  L_0\phi(x)=b_x[\phi(x+1)-\phi(x)]+ d_x[\phi(x-1)-\phi(x)].
  $$

  For all $x\in E$, we compute
  \begin{align*}
    L_0V(x) & =b_{x}\frac{1}{ d_{x+1}\pi_{x+1}}\sum_{k\geq x+1}W(k)\pi_k- d_{x}\frac{1}{ d_{x}\pi_{x}}\sum_{k\geq x}W(k)\pi_k
    \\ & =\frac{1}{\pi_x}\sum_{k\geq x+1}W(k)\pi_k-\frac{1}{\pi_x}\sum_{k\geq x}W(k)\pi_k
    \\ & =-W(x).
  \end{align*}
  Since $L_0\mathbbm{1}_E(x)= -d_1\mathbbm{1}_{x=1}$, $a-\kappa=o(W)$ and $V$ is bounded, we deduce that, for all probability measure
  $\mu$ on $E$,
  \begin{align*}
    \mu(LV)-\mu(L\11_E)\mu(V) & =\mu(L_0V)-\mu( a V)+ d_1\mu(\{1\})\mu(V)+\mu( a)\mu(V) \\
    & \leq-\mu(W+o(W))-\mu(\kappa V)+ d_1\|V\|_\infty+\mu(\kappa)\mu(V) \\
    & \leq-\mu(W+o(W))+ d_1\|V\|_\infty \\ & \leq A-\frac{1}{2}\mu(W)
  \end{align*}
  for some constant $A$, where the fact that $\mu(\kappa)\mu(V)\leq\mu(\kappa V)$ follows from FKG inequality since $V$ and $\kappa$
  are non-decreasing.
\end{proof}

\section{Multi-dimensional birth and death processes with unbounded extinction rate}
\label{sec:multi-dim}

\subsection{Domination by a one-dimensional birth and death process}

\label{sec:multi-d-dom}

It is easy to obtain multidimensional extensions of Theorem~\ref{thm:PNM-1D} provided that the total number of individuals is
dominated by a birth and death process which comes down from infinity. For example, let us consider a birth and death process
$(X_t,t\geq 0)$ in $\mathbb{Z}_+^r$, $r\geq 1$, with jump rate $  b_i(x)\geq 0$ from $x$ to $x+e_i$ (with $e_i=(0,\ldots,0,1,0\ldots,0)$ where
the 1 is at the $i$-th coordinate), death rate $d_i(x)\geq 0$ from $x$ to $x-e_i$ and catastrophe rate $ a(x)$, absorbed at
$\d=0$. 
Assume that there exist two positive functions $\bar{b}$ and $\underline{d}$ on $\mathbb{N}$  and a
nonnegative, nondecreasing function $\kappa$ on $\mathbb{N}$ such that
\begin{align}
\label{eq:bd-rate-majmin}
\sum_{i=1}^r  b_i(x)\leq\bar{  b}(|x|),\quad \sum_{i=1}^r d_i(x)\geq\underline{d}(|x|)
\end{align}
and
$$
a(x)=\kappa(|x|)+o(W(|x|))\quad\mbox{when }|x|\rightarrow+\infty,
$$
where $|x|=x_1+\ldots+x_n$ and where the function $W$ satisfies the conditions of Theorem~\ref{thm:PNM-1D}
for the one-dimensional birth and death process with birth rates $\bar{b}$ and death rates $\underline{d}$.

Then, the proof of Theorem~\ref{thm:PNM-1D} extends easily to this situation and hence the conclusion of Theorem~\ref{thm:Lyap} holds
true.

\begin{exa}
\label{exa:muli-dim-dom1}
Let us consider a multi-dimensional birth and death process whose only absorption point in $\d=(0,\ldots,0)$. More precisely, we
consider the multi-dimensional birth and death process with mutation such that there exist positive constants $\beta_i$, $\delta_i$,
$m_{ij}$ and $c_{ij}$ for all $i,j\in\{1,\ldots,r\}$ such that, for all $x=(x_1,\ldots,x_r)\in \Z_+^r\setminus\{0\}$,
\begin{align*}
b_i(x)&=\beta_i\,x_i+\sum_{j=1,j\neq i}^r m_{ij}\,x_j,\\
d_i(x)&=\delta_i\,x_i+c_{ii}x_i(x_i-1)+\sum_{j=1,j\neq i}^r c_{ij}\,x_ix_j.
\end{align*}
In this model, $\beta_i$ and $\delta_i$ represent respectively the individual birth and death rates of an individual of type $i$,
while $m_{ij}$ represents the individual birth rate of an individual of type $i$ due to the reproduction with mutation of an
individual of type $j$ and $c_{ij}$ represents the individual death rate of an individual of type $i$ due to the competition of an
individual of type $j$. As above, we denote by $a(x)$ the catastrophe rate of the process at point $x\in\Z_+^d$.

Setting $\bar{\beta}=\max_i \beta_i+\sum_{j\neq i}m_{ij}$, $\bar{\delta}=\min_i \delta_i$ and $\bar{c}=\min_{ij} c_{ij}$, we observe
that the condition~\eqref{eq:bd-rate-majmin} is satisfied for $\bar{b}(|x|)=\bar{\beta}|x|$ and
$\bar{d}(|x|)=\bar{\delta}|x|+\bar{c}|x|(|x|-1)$. Hence, using the conclusion of Example~\ref{ex:exa-log-1D}, we deduce that the
conclusion of Theorem~\ref{thm:Lyap} holds true as soon as $|a(x)-\kappa(|x|)|= O(|x|^{1-\varepsilon})$ for some $\varepsilon>0$ and
some non-decreasing non-negative function $\kappa$ on $\N$.
\end{exa}

\begin{exa}
\label{exa:muli-dim-dom2}
Let us now consider a situation where the process is absorbed when it hits $\partial:=\{x=(x_1,\ldots,x_r)\in\Z_+^r,\,\exists i\text{
  s.t. }x_i=0\}$ and which can only be absorbed from states in $\{x=(x_1,\ldots,x_r)\in\Z_+^r,\,\exists i\text{ s.t. }x_i=1\}$. This
corresponds to a multi-dimensional birth and death process without catastrophe nor mutation absorbed when one of the type disappears.
To be consistent with our notation, we actually take $E=\mathbb{N}^r=\{1,2,\ldots\}^r$ and $\d$ a point which does not belong to $E$,
and we assume that absorptions are due only to the rates $a(x)$ for $x$ such that $x_i=1$ for some $1\leq i\leq r$. More precisely,
we assume that there exist positive constants $\beta_i$, $\delta_i$ and $c_{ij}$ for all $i,j\in\{1,\ldots,r\}$ such that, for all
$x=(x_1,\ldots,x_r)\in \Z_+^r\setminus\{0\}$,
\begin{align*}
b_i(x)&=\beta_i\,x_i,\\
d_i(x)&=\11_{x_i\neq 1}\left(\delta_ix_i+c_{ii}x_i(x_i-1)+\sum_{j=1,j\neq i}^r c_{ij}\,x_ix_j\right),\\
a(x)&=a(x)\11_{\exists i\text{ s.t. }x_i=0}.
\end{align*}

Setting $\bar{\beta}=\max_i \beta_i$, $\bar{\delta}=a(1,\ldots,1)\wedge\inf_i\delta_i$ and $\bar{c}=\min_{ij} c_{ij}$, one can check
that the condition~\eqref{eq:bd-rate-majmin} is satisfied for $\bar{b}(|x|)=\bar{\beta}|x|$ and $
\bar{d}(|x|)=\bar{\delta}+\bar{c}|x|(|x|-r)_+$. Hence, using the same calculation as in Example~\ref{ex:exa-log-1D} (see also
Remark~\ref{rem:arg-log-gener}), we deduce that the conclusion of Theorem~\ref{thm:Lyap} holds true as soon as
$a(x)=O(|x|^{1-\varepsilon})$ for some $\varepsilon>0$.
\end{exa}

\subsection{Strong intra-specific competition cases}
\label{sec:multi-d-strong-intra}

The methods described in the previous subsection are particularly well suited to the study of a multi-type population conditioned to
global non-extinction, \textit{i.e.} conditioned to the survival of at least one type in the population. In the present subsection,
we are interested in the quasi-stationary behavior of a multi-type population process subject to catastrophic events and conditioned
to the survival of all the types in the population.

More precisely, we consider a multi-dimensional birth and death process $(X_t,t\geq 0)$ taking values in $\NN^r\cup\{\d\}$ for some
$r\geq 1$, absorbed at $\d$ and whose coefficients are given, for all $x=(x_1,\ldots,x_r)\in\NN^r$ by
\begin{align}
b_i(x)&=\beta_i(x)\,x_i, \notag \\
d_i(x)&=\11_{x_i\neq 1}\left(\delta_i(x)x_i+c_{ii}(x)\,x_i(x_i-1)+\sum_{j=1,j\neq i}^r c_{ij}(x)\,x_ix_j\right),\notag \\
a(x)&=\alpha(x)+\sum_{i=1}^r\11_{x_i=1}\left(\delta_i(x)+\sum_{j=1,j\neq i}^r c_{ij}(x)\,x_j\right), \label{eq:def-a}
\end{align}
where $\alpha$, $\beta_i$, $\delta_i$ and $c_{ij}$ are functions from $\NN^r$ to $\RR_+$. Similarly to
examples~\eqref{exa:muli-dim-dom1} and~\eqref{exa:muli-dim-dom2}, this is a model for a population where $\beta_i(x)$ and
$\delta_i(x)$ represent respectively the individual birth and death rates of an individual of type $i$ in population $x$, while
$c_{ij}(x)$ represents the individual death rate of an individual of type $i$ in a population $x$ due to the competition of an
individual of type $j$. Note that the value of $a(x)$ is divided into two parts : $\alpha(x)$ represents the rate of catastrophic
events annihilating the whole population, while the other part is the rate at which the process would jump to $\ZZ_+\setminus\NN_+$
due to an individual death. Both phenomena are gathered into only one jump rate $a(x)$ from $x$ to $0$ because we are only interested
in the behaviour of the process conditioned not to hit the boundary $\ZZ_+\setminus\NN_+$.

We emphasize that this natural setting leads to a new difficulty that cannot be handled using the methods of the previous subsection.
Indeed, beside the dependence on $x$ of all the individual rates, even in the case where $\beta_1$, $\delta_i$ and $c_{ij}$ are
independent of $x$, the growth of $a(x)$ due to the sum in~\eqref{eq:def-a} does not fit in the study of Examples~\ref{exa:muli-dim-dom1} and~\ref{exa:muli-dim-dom2}.

We will make the following natural assumption that there exist constants $\bar b$, $\bar d$ and $\underline{c}$ in $(0,\infty)$, such that, for all $n\in\NN^r$ and $i\in\{1,\ldots,r\}$,
    \begin{align}
    \label{eq:bounds-on-coeffs}
    0<\beta_i(x)\leq \bar \beta,\quad 0\leq \delta_i(x)\leq\bar \delta,\quad c_{ii}(x)\geq\underline{c}.
    \end{align}
Moreover, we make the following important assumption  which indicates that the intra-specific competition dominates the inter-specific competition. 
    
\paragraph{Assumption H1.} There exist $\eta\in(0,1)$ such that
    \begin{equation}
      \label{eq:hyp}
\sum_{1\leq j\neq k\leq r} c_{jk}(x)x_j|x| 
      \leq (1-\eta)\sum_{i=1}^r c_{ii}(x)x_i(x_i-1),\text{ for \ $|x|$ large enough.} 
    \end{equation}

In biology, this assumption is particularly relevant to model a multi-type population where individuals survive by consuming a type-specific resource. Indeed, in this situation, one would expect to have $c_{jk}(x) \ll c_{ii}(x)$, for all $j\neq k$, all $i\in\N^r$ and all $|x|$ large enough.

Let us introduce, for all $x\in\NN^r$,
\begin{align*}
\kappa^+(x)=\sup_{y\in\NN^r,|y|\leq |x|} \alpha(y)
\end{align*}
and
\begin{align*}
\kappa^-(x)=\inf_{y\in\NN^r,|y|\geq |x|} \alpha(y).
\end{align*}
We also set $\text{Osc}(\alpha)(x)=\kappa^+(x)-\kappa^-(x)$. The following assumption states that the oscillations of $\alpha$ are sub-linear at infinity. In particular, we make no assumption on the growth of $\alpha$ itself. Also, we emphasize that the following assumption is trivially fulfilled when $\alpha=0$.

\paragraph{Assumption H2.} There exists $\eta'\in(0,\eta)$ such that
\begin{align*}
\text{Osc}(\alpha)(x)=o(|x|^{\eta'})\text{ when $|x|\rightarrow+\infty$}.
\end{align*}

Note that the larger $\eta$ can be chosen, the weaker Assumption H2 is. In fact, if $\sum_{1\leq j\neq k\leq r}
c_{jk}(x)=o(c_{ii}(x))$ when $|x|\rightarrow+\infty$ for all $1\leq i\leq r$, we can take any $\eta'<1$ in Assumption H2.

We can now state the following result. 
\begin{thm}
  \label{thm:one}
  Assume that~\eqref{eq:bounds-on-coeffs},~H1 and~H2 hold, then the conclusion of Theorem~\ref{thm:Lyap} holds true.
\end{thm}

\begin{rem}
  \label{rem:thm-one-plus-general}
  A slight modification of the proof shows that one can replace Assumption~(H1) by 
  \begin{align}
    \label{eq:alt-hyp}
    \sum_{j=1}^r \frac{x_j}{|x|}\mathbbm{1}_{x_j\neq 1}\sum_{k=1}^r c_{jk}(x)(x_k-\mathbbm{1}_{k=j})\geq \frac{1}{1-\eta}
    \sum_{j=1}^r \mathbbm{1}_{x_j=1}\sum_{k=1}^r c_{jk}(x)(x_k-\mathbbm{1}_{k=j})
  \end{align}
  for $|x|$ large enough. 
While it might be less biologically relevant, this assumption allows to handle situations that do not fit into the strong intra-specific assumption.
For example, when $r=2$, $c_{12}(x)=x_1$ and $c_{21}(x)=c_{11}(x)=c_{22}(x)=1$, the inequality~\eqref{eq:alt-hyp} holds true for large values of $|x|$.  
\end{rem}

\begin{proof}[Proof of Theorem~\ref{thm:one}]
Let us first consider the case where $\alpha(x)=0$ for all $x\in\NN^r$.
Fix $\varepsilon\in(0,1)$ and define for
all $x\in\NN^r$
$$
V(x)=\sum_{j=1}^{|x|}\frac{1}{j^{1+\varepsilon}}.
$$
For all $x, y\in\NN^r$ such that $|x|\leq|y|$, we have in
particular the inequality
\begin{multline}
  \label{eq:bound-V}
  \frac{1}{\varepsilon}\left(\frac{1}{(|x|+1)^\varepsilon}-\frac{1}{(|y|+1)^\varepsilon}\right)=\int_{|x|+1}^{|y|+1}\frac{dz}{z^{1+\varepsilon}} \\
  \leq V(y)-V(x)
  \leq\int_{|x|}^{|y|}\frac{dz}{z^{1+\varepsilon}}=\frac{1}{\varepsilon}\left(\frac{1}{|x|^\varepsilon}-\frac{1}{|y|^\varepsilon}\right).
\end{multline}
\medskip

For any $x\in \NN^r$, we have
\begin{align*}
 LV(x)&=\sum_{i=1,\,x_i\neq 1}^r \frac{\beta_i(x)x_i}{(|x|+1)^{1+\varepsilon}}-\frac{\delta_i(x)x_i+\sum_{j=1}^r c_{ij}(x)
         x_i (x_j-\mathbbm{1}_{j=i})}{|x|^{1+\varepsilon}} \\
 &\quad\quad+\sum_{i=1,\,x_i=1}^r\frac{\beta_i(x)}{(|x|+1)^{1+\varepsilon}}-\left[\delta_i(x)+\left(\sum_{j=1,\ j\neq i}^r c_{ij}(x)
       x_j\right)\right]V(x)
\end{align*}
Using the inequality assumption~\eqref{eq:bounds-on-coeffs},
\begin{align*}
 LV(x)\leq\bar \beta
    |x|^{-\varepsilon}-\sum_{i=1,\,x_i\neq 1}^r\frac{ c_{ii}(x)
         x_i (x_i-1)}{|x|^{1+\varepsilon}}
 -\sum_{i=1,\,x_i=1}^r\left(\delta_i(x)+\sum_{j=1,\ j\neq i}^r c_{ij}(x) x_j\right)
    V(x). 
\end{align*}
Moreover, for any probability measure $\mu$ on $\NN^r$,
$$
-\mu(V)\mu(L\mathbbm{1}_{\NN^r})\leq\|V\|_\infty\sum_{x\in\NN^r}\mu(x)
\left(\sum_{i=1,x_i=1}^r\delta_i(x)+\sum_{j=1,\ j\neq i}^r c_{ij}(x) x_j\right).
$$
The last two equations imply
$$
\mu(LV) -\mu(V)\mu(L\mathbbm{1}_{\NN^r})\leq \sum_{x\in\NN^r}\mu(x)\left[\bar \beta
  |x|^{-\varepsilon}-\sum_{i=1,\,x_i\neq 1}^r\frac{ c_{ii}(x)
         x_i (x_i-1)}{|x|^{1+\varepsilon}}\right]+(*),
$$
where, by~\eqref{eq:bound-V},
\begin{align*}
  (*) & :=\sum_{x\in\NN^r}\mu(x)\sum_{i=1,x_i=1}^r \left[\delta_i(x)+\sum_{j=1,\ j\neq i}^r
      c_{ij}(x) x_j\right]\left(\|V\|_\infty-V(x)\right) \\ 
  & \leq\sum_{x\in\NN^r}\mu(x)\sum_{i=1,x_i=1}^r\frac{\bar \delta+\sum_{j\neq i}c_{ij}(x)x_j}{\varepsilon|x|^\varepsilon} 
\end{align*}

Using Assumption~H1, we see that for all $\varepsilon\in(1-\eta,1)$, there exist some positive constants $A,B,B',C,D>0$ independent of
$x$ (but dependent on $\varepsilon$) and such that
\begin{align*}
  \mu(LV) -\mu(V)\mu(L\mathbbm{1}_{\NN^r}) & \leq \sum_{x\in\NN^r}\mu(x)
  \left(A|x|^{-\varepsilon}-B\sum_{i=1,\,x_i\neq 1}^r\frac{ c_{ii}(x)
         x_i (x_i-1)}{|x|^{1+\varepsilon}}\right) \\
  & \leq \sum_{x\in\NN^r}\mu(x)
  \left(A|x|^{-\varepsilon}-B'\underline{c}|x|^{1-\varepsilon}\right) \\
  & \leq C-D\sum_{x\in\NN^r}|x|^{1-\varepsilon}\mu(x)=C-D\mu(W),
\end{align*}
where we used the fact that $x_i\geq |x|/r$ for at least one index $i$ and where $W:x\in\NN^r\mapsto |x|^{1-\varepsilon}$. This entails that the assumptions of Theorem~\ref{thm:Lyap} are satisfied and concludes the proof of Theorem~\ref{thm:one} in the case where $\alpha(x)=0$ for all $x\in\NN^r$.

Assume now that $\alpha(x)\neq 0$ for some $x\in\NN^r$. In this case, the same calculation as above leads to
\begin{align*}
\mu(LV)-\mu(V)\mu(L\11_{\NN^r})\leq C-D\sum_{x\in \NN^r}|x|^{1-\varepsilon}\mu(x)-\mu(\alpha V)+\mu(V)\mu(\alpha).
\end{align*}
We obtain
\begin{align*}
-\mu(\alpha V)+\mu(V)\mu(\alpha)&\leq -\mu((\alpha-\kappa^-)V) + \mu(V)\mu(\alpha-\kappa^-)\\
&\leq \mu(V)\mu(\alpha-\kappa^-)\leq \|V\|_\infty \mu(\text{Osc}(\alpha))
\end{align*}
where we used the fact that $-\mu(\kappa^- V)+\mu(V)\mu(\kappa^-)\leq 0$ by the FKG inequality, since both $\kappa^-(x)$ and $V(x)$
are non-decreasing with $|x|$, and the fact that $\alpha-\kappa^-\geq 0$. Now, by assumption H2, since we may choose
$\varepsilon\in(1-\eta,1-\eta')$,
\begin{align*}
\text{Osc}(\alpha)(x)=o(|x|^{1-\varepsilon})\text{ when $|x|\rightarrow+\infty$},
\end{align*}
and we deduce that there exist some positive constants $A',B'$ such that
\begin{align*}
\mu(LV)-\mu(V)\mu(L\11_{\NN^r})\leq A'-B'\mu(W),
\end{align*}
for the same norm-like function $W$ as above. This and Theorem~\ref{thm:Lyap} allow us to conclude the proof of
Theorem~\ref{thm:Lyap}.
\end{proof}

\bibliographystyle{abbrv}
\bibliography{biblio-bio,biblio-denis,biblio-math,biblio-math-nicolas}

\end{document}